\documentclass[12pt,oneside]{amsart}
\usepackage{tkz-euclide}
\usepackage{tkz-euclide}
\usetikzlibrary{graphs}
\usepackage{geometry}                
\usepackage{float}
\usepackage{graphicx}
\usepackage{amssymb}

\usepackage{amsmath,amsthm,amscd,amssymb}
\usepackage{latexsym}
\usepackage[colorlinks,citecolor=red,pagebackref,hypertexnames=false]{hyperref}
\usepackage{ulem}
\usepackage{soul}
\numberwithin{equation}{section}
\theoremstyle{plain}
\newtheorem{theorem}{Theorem}[section]
\newtheorem{lemma}[theorem]{Lemma}
\newtheorem{corollary}[theorem]{Corollary}

\theoremstyle{definition}
\newtheorem{definition}[theorem]{Definition}

\newtheorem{example}[theorem]{Example}

\theoremstyle{remark}
\newtheorem{remark}[theorem]{Remark}

\newtheorem{case[theorem]}{Case}

\def\bc{\begin{corollary}}
\def\ec{\end{corollary}}
\def\be{\begin{equation}}
\def\ee{\end{equation}}
\def\bast{\begin{eqnarray*} }
\def\east{\end{eqnarray*} }
\def\bea{\begin{eqnarray}}
\def\eea{\end{eqnarray}}

\def\qed{\hfill\Box\smallskip}

\def\R{\mathbb R}

\def\R{\Bbb R}

\def\({\left(}
\def\){\right)}
\def\[{\left[}
\def\]{\right]}
\def\<{\left\langle}
\def\>{\right\rangle}

\author{A. Iosevich}
\address[A. Iosevich]{Department of Mathematics, University of Rochester, Rochester, NY}
\email{iosevich@math.rochester.edu}

\author{E. Palsson}
\address[E. Palsson]{Department of Mathematics, Virginia Tech, Blacksburg, VA}
\email{palsson@vt.edu}

\author{E. Wyman}
\address[E. Wyman]{Department of Mathematics and Statistics, Binghamton University, Binghamton, NY}
\email{ewyman@math.binghamton.edu}

\author{Y. Zhai}
\address[Y. Zhai]{Institute for Advanced Study in Mathematics, Harbin Institute of Technology, China}
\email{yzhai@hit.edu.cn}

\thanks{The first listed author is supported in part by the National Science Foundation grant no. HDR TRIPODS - 1934962. The third author is supported in part by the National Science Foundation grant no. DMS-2204397.}
\thanks{}
\begin{document}

\title{Multi-linear forms, structure of graphs and Lebesgue spaces}

\maketitle

\begin{abstract} Consider the operator
$$T_Kf(x)=\int_{{\mathbb R}^d} K(x,y) f(y) dy,$$ where $K$ is a locally integrable function or a measure. The purpose of this paper is to study the multi-linear form 
$$ \Lambda^K_G(f_1, \dots, f_n)=\int \dots \int \prod_{ \{(i,j): 1 \leq i<j \leq n; E(i,j)=1 \} } K(x^i,x^j) \prod_{i=1}^n f_i(x^i) dx^i, $$ where $G$ is a connected graph on $n$ vertices, $E$ is the edge map on $G$, i.e $E(i,j)=1$ if and only if the $i$'th and $j$'th vertices are connected by an edge, $K$ is the aforementioned kernel, and $f_i: {\mathbb R}^d \to {\mathbb R}$, measurable.

This paper establishes multi-linear inequalities of the form 
$$ \Lambda^K_G(f_1,f_2, \dots,f_n) \leq C {||f_1||}_{L^{p_1}({\mathbb R}^d)} {||f_2||}_{L^{p_2}({\mathbb R}^d)} \dots {||f_n||}_{L^{p_n}({\mathbb R}^d)}$$ and determines how the exponents depend on the structure of the kernel $K$ and the graph $G$. 
\end{abstract}  

\maketitle

\section{Introduction}

\vskip.125in 

A ubiquitous object in analysis is an operator of the form 
$$T_Kf(x)=\int_{{\mathbb R}^d} K(x,y) f(y) dy,$$ where $K$ is a locally integrable function or a measure. The purpose of this paper is to study the multi-linear form 
\begin{equation} \label{multilinearmama} \Lambda^K_G(f_1, \dots, f_n)=\int \dots \int \prod_{ \{(i,j): 1 \leq i<j \leq n; E(i,j)=1 \} } K(x^i,x^j) \prod_{i=1}^n f_i(x^i) dx^i \end{equation} where $G$ is a connected graph on $n$ vertices, $E$ is the edge map on $G$, i.e $E(i,j)=1$ if and only if the $i$'th and $j$'th vertices are connected by an edge, $K$ is the aforementioned kernel, and $f_i: {\mathbb R}^d \to {\mathbb R}$, measurable. Note that the form is well-defined if 
\begin{equation} \label{localintegrability} \prod_{ \{(i,j): 1 \leq i<j \leq n; E(i,j)=1 \} } K(x^i,x^j) \end{equation} is locally integrable in $x^1, \dots, x^n$. In the case when $K$ is a measure, we consider 
$$\Lambda^{K,\epsilon}_G(f_1, \dots, f_n)=\int \dots \int \prod_{ \{(i,j): 1 \leq i<j \leq n; E(i,j)=1 \} } K^{\epsilon}(x^i,x^j) \prod_{i=1}^n f_i(x^i) dx^i$$ and ask for local integrability with constants independent of $\epsilon$, where $K^{\epsilon}=K*\rho_{\epsilon}$, where $\rho$ is non-negative cutoff function with $\int \rho =1$. 

Determining for which connected graphs $G$ the expression (\ref{localintegrability}) is locally integrable already poses some interesting challenges for a given kernel $K$. An analogous problem was studied in \cite{BIKP23} in the setting of vector spaces over finite fields. In the Euclidean case, when $K(x,y)={|x-y|}^{-s}$, $s>0$, and $G$ is the complete graph, this problem was studied in (\cite{SWY19}). There are also some connections with the Brascamp-Lieb-type inequalities studied by many authors over the past several decades. See Remark \ref{rem: brascamp-lieb} below for more details. 

\vskip.125in 

More generally, we may consider 
\begin{equation} \label{Multilinearfractalmama}\Lambda^{K,\mu}_G(f_1, \dots, f_n)=\int \dots \int \prod_{ \{(i,j): 1 \leq i<j \leq n; E(i,j)=1 \} } K(x^i,x^j) \prod_{i=1}^n f_i(x^i) d\mu(x^i), \end{equation} where $\mu$ is a compactly supported Borel measure. 

These types of forms naturally arise in several contexts, and we are particularly motivated by the study of finite point configurations (see e.g. \cite{BIT15}, \cite{EIT11}, \cite{Fal86}, \cite{IKSTU19}, \cite{IT19}, \cite{GIT22}, \cite{M95}, \cite{Mat2015} and the references contained therein), where the goal is to show that if the underlying set is sufficiently large, in the sense of Lebesgue density, Hausdorff dimension, etc., then a given geometric configuration, such as every sufficiently large distance, or an equilateral triangle can be realized inside this set in a suitable sense. This version of the problem will be studied systematically in the sequel. See also Section \ref{future} below where we outline further steps in this direction. 

Similar definitions can be created if the linear operator $T_K$ is replaced by a multi-linear operator, and the graph $G$ is replaced by a hyper-graph. We shall expand upon this idea in Section \ref{future} below. 

\vskip.125in 

The class of operators we are going to consider in this paper are linear $L^p$-improving operators. We shall need the following definition. 

\begin{definition} \label{lpimprovingdef} Let $T_K$ be as above. We say that $T_K$ is $L^p$-improving (for a particular 
$p \in (1,\infty)$) if there exists $q>p$ such that $T_K:L^p({\mathbb R}^d) \to L^q({\mathbb R}^d)$. \end{definition} 

\begin{definition} \label{universallylpimprovingdef} Let $T_K$ be as above. We say that $T_K$ is {\it universally} $L^p$-improving, if for every $q \in (1,\infty)$, there exists $1<p<q$ such that 
$$ T_K: L^p({\mathbb R}^d) \to L^q({\mathbb R}^d).$$ 
\end{definition} 

\begin{definition} \label{multilinearlpimproving} Let $U$ be an $m$-linear operator. We say that $U$ is $L^p$-improving if there exist $r, p_1, \dots, p_m$ in $[1, \infty]$ such that 
\begin{equation} \label{sumisbigger} \sum_{i=1}^m \frac{1}{p_i}>\frac{1}{r}, \end{equation} and 
$$ {||U(f_1, \dots, f_m)||}_{L^r({\mathbb R}^d)} \leq C {||f_1||}_{L^{p_1}({\mathbb R}^d)} \cdot {||f_2||}_{L^{p_2}({\mathbb R}^d)} \dots \cdot {||f_m||}_{L^{p_m}({\mathbb R}^d)}.$$
\end{definition} 

\vskip.125in 

\begin{definition} \label{universallymultilinearlpimproving} Let $U$ be an $m$-linear operator. We say that $U$ is universally $L^p$-improving if for every $r >1$ there exist $p_1, \dots, p_m$ in $[1, \infty]$ such that (\ref{sumisbigger}) holds, and 
$$ {||U(f_1, \dots, f_m)||}_{L^r({\mathbb R}^d)} \leq C {||f_1||}_{L^{p_1}({\mathbb R}^d)} \cdot {||f_2||}_{L^{p_2}({\mathbb R}^d)} \dots \cdot {||f_m||}_{L^{p_m}({\mathbb R}^d)}.$$
\end{definition} 

\vskip.125in 

\begin{definition} \label{GformLpdef} Let $T_K$ and $\Lambda^K_G$ be as above, with $T_K$ $L^p$-improving. We say that 
$\Lambda^K_G$ is $L^p$-improving if there exist $p_1,p_2, \dots, p_n$, such that 
$$ \frac{1}{p_1}+\frac{1}{p_2}+\dots+\frac{1}{p_n}>1,$$ and 
\begin{equation} \label{mamabound} \Lambda^K_G(f_1,f_2, \dots,f_n) \leq C {||f_1||}_{L^{p_1}({\mathbb R}^d)} {||f_2||}_{L^{p_2}({\mathbb R}^d)} \dots {||f_n||}_{L^{p_n}({\mathbb R}^d)}, \end{equation} 
\end{definition} 

\vskip.125in 

The questions we ask are the following: 

\begin{itemize} 

\vskip.125in 

\item {\bf Question 1:} Given a kernel $K$, with $T_K$ $L^p$-improving, for which connected graphs $G$ does the multi-linear form $\Lambda_G$ satisfy better than Holder bounds? 

\vskip.125in 

\item {\bf Question 2:} How does the structure of the graph $G$ influence the $L^p$ norms of the functions needed to bound $\Lambda^K_G(f_1, \dots, f_n)$? 

\vskip.125in 

\item{\bf Question 3:} Is it true that if $G$ is a subgraph of $G'$ and $\Lambda_{G'}^K$ satisfies (\ref{mamabound}), then $\Lambda_G^K$ also satisfies (\ref{mamabound}). 
\end{itemize} 

\vskip.25in 

In \cite{BIKP23}, the authors showed in the context of ${\mathbb F}_q^d$, the $d$-dimensional vector space over the finite field with $q $ elements, that if $K$ is the indicator function of a sphere of non-zero radius, then while the first graph above (the $4$-cycle) is a subgraph of the second graph (the $4$-cycle with an additional edge), the answer to Question 3 is negative. See Figure \ref{twoquads} below. We were not able to prove this in the Euclidean setting and hope to address this in the sequel. More details on work in progress and possible extensions of the results in this paper can be found in Section \ref{future}. 

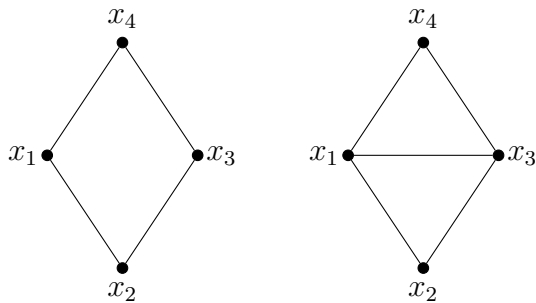
\begin{figure}[H]
\begin{tikzpicture}[node distance=2cm]

  \tikzstyle{every node}=[circle, draw, fill=black, inner sep=0pt, minimum width=4pt]

  \node[label=left:$x_1$] (x1) at (0,0) {};
  \node[label=below:$x_2$] (x2) at (1,-1.5) {};
  \node[label=right:$x_3$] (x3) at (2,0) {};
  \node[label=above:$x_4$] (x4) at (1,1.5) {};

  \node[label=left:$x_1$] (x7) at (4,0) {};
  \node[label=below:$x_2$] (x8) at (5,-1.5) {};
  \node[label=right:$x_3$] (x9) at (6,0) {};
  \node[label=above:$x_{4}$] (x10) at (5,1.5) {};

  \draw (x1) -- (x2);
  \draw (x2) -- (x3);
  \draw (x3) -- (x4);
  \draw (x4) -- (x1);

  \draw (x7) -- (x8);
  \draw (x8) -- (x9);
  \draw (x9) -- (x10);
  \draw (x10) -- (x7);
  \draw (x7) -- (x9); 
   
\end{tikzpicture}
\caption{A rhombus and a rhombus with an extra edge} 
\label{twoquads}
\end{figure}






\vskip.125in 

\subsection{Results with a general kernel $K$}

Our first result says that for an arbitrary kernel $K$, the $L^p$-improving property of $T_K$ implies an $L^p$-improving property of $\Lambda^K_G$ if $G$ is a tree (see Figure \ref{trees} below). 

\begin{theorem} \label{maintree} (Trees) Suppose that $T_K$ is universally $L^p$-improving, and $G$ is a connected tree graph. Then $\Lambda^K_G$ is $L^p$-improving. \end{theorem} 

\vskip.125in 

\begin{figure}[H]
\begin{tikzpicture}[node distance=2cm]

  \tikzstyle{every node}=[circle, draw, fill=black, inner sep=0pt, minimum width=4pt]

  \node[label=left:$x_1$] (x1) at (0,0) {};
  \node[label=left:$x_2$] (x2) at (-0.5,-1) {};
  \node[label=right:$x_3$] (x3) at (0.5,-1) {};
  \node[label=right:$x_4$] (x4) at (0.5,-2) {};
  \node[label=left:$x_5$] (x5) at (0,-3) {};
  \node[label=right:$x_6$] (x6) at (1,-3) {};

  \node[label=left:$x_1$] (x7) at (2,0) {};
  \node[label=right:$x_2$] (x8) at (2.5,-1) {};
  \node[label=right:$x_3$] (x9) at (3,-2) {};
  \node[label=right:$x_{4}$] (x10) at (3.5,-3) {};

  \draw (x1) -- (x2);
  \draw (x1) -- (x3);
  \draw (x3) -- (x4);
  \draw (x4) -- (x5);
  \draw (x4) -- (x6);

  \draw (x7) -- (x8);
  \draw (x8) -- (x9);
  \draw (x9) -- (x10);

\end{tikzpicture}
\caption{Tree graphs} 
\label{trees}
\end{figure}
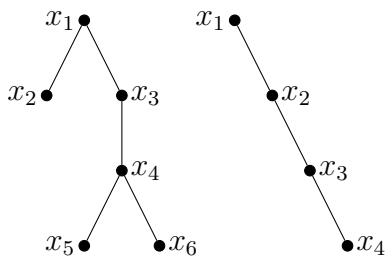 

\vskip.125in 

Our second result says that if we trim trees from the graph $G'$, obtaining a graph $G$, and $\Lambda_G^K$ is $L^p$-improving, then $\Lambda_{G'}^K$ is $L^p$-improving. See Figure \ref{treeandtriangle} below. To state this precisely, we need another definition. 

\begin{definition} \label{deforestation} Let $G$ be a connected proper sub-graph of a graph $G'$. We say that $G$ is a contraction of $G'$ if every vertex of $G$ is a root of a (possibly empty) tree with the remaining vertices (if any) in the complement of $G$ in $G'$, such that the trees with roots at the vertices of $G$ are disjoint. \end{definition} 

\begin{theorem} \label{treeremoval} (Tree removal) Suppose that a connected graph $G$ is a contraction of a connected graph $G'$. Then $\Lambda_{G'}^K$ is $L^p$-improving if $\Lambda_G^K$ is $L^p$-improving. 
\end{theorem} 

\vskip.125in 

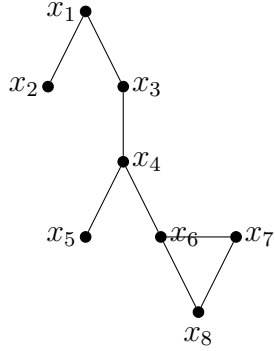
\begin{figure}[H]
\begin{tikzpicture}[node distance=2cm]

  \tikzstyle{every node}=[circle, draw, fill=black, inner sep=0pt, minimum width=4pt]

  \node[label=left:$x_1$] (x1) at (0,0) {};
  \node[label=left:$x_2$] (x2) at (-0.5,-1) {};
  \node[label=right:$x_3$] (x3) at (0.5,-1) {};
  \node[label=right:$x_4$] (x4) at (0.5,-2) {};
  \node[label=left:$x_5$] (x5) at (0,-3) {};
  \node[label=right:$x_6$] (x6) at (1,-3) {};

  \draw (x1) -- (x2);
  \draw (x1) -- (x3);
  \draw (x3) -- (x4);
  \draw (x4) -- (x5);
  \draw (x4) -- (x6);

  \node[label=right:$x_7$] (x7) at (2,-3) {};
  \node[label=below:$x_8$] (x8) at (1.5,-4) {};

  \draw (x6) -- (x7);
  \draw (x7) -- (x8);
  \draw (x8) -- (x6);

\end{tikzpicture}
\caption{Triangle with a tree added on} 
\label{treeandtriangle}
\end{figure}

Our next result allows us to string together graphs that share a single vertex. See Figure \ref{twographs} below. We start with a definition.

\vskip.125in 

\begin{definition}\label{FormUniversalLpImproving}
Let $G$ be a connected graph on $n$ vertices and $K$ a locally integrable kernel. 
\begin{itemize}
\item[(i)] We say the form $\Lambda_{G}^K$ has a non-trivial estimate at a vertex $i$ if $\Lambda_{G}^K$ is bounded by
$$ \Lambda_{G}^{K}(f_1,\ldots,f_{i},\ldots,f_n) \leq C \|f_1\|_{p_1} \ldots \|f_i\|_{p_i}\ldots\|f_n\|_{p_n} $$
with $\frac{1}{p_1}+\ldots+\frac{1}{p_i}+\ldots+\frac{1}{p_n}\geq 1$ where $1\leq p_i<\infty$ while $1\leq p_j \leq \infty$ if $j\neq i$.
\vskip.125in 
\item[(ii)] We say that the form $\Lambda_G^K$ is universally $L^p$ improving if each of the $(n-1)$-linear operators $T_{i}$ that arise from
$$\Lambda_G^K(f_1,\ldots,f_n) = \int T_{i}(f_1,\ldots,f_{i-1},f_{i+1},\ldots,f_n)(x_i) f_{i}(x_i) dx_i ,$$
$1 \leq i \leq n$, are universally $L^p$ improving for each $1<r<\infty$.
\end{itemize}
 
\end{definition}

\begin{theorem} \label{joininggraphs} Let $G_1, G_2$ be connected graphs, with $n_1$ and $n_2$ vertices, respectively. Suppose that $G_1$ and $G_2$ have a single vertex in common and let $G=G_1 \cup G_2$ in the sense that we take the union of the vertex sets and the edge relations remain the same as before. Let $K$ be a locally integrable kernel such that at least one of
$$\Lambda_{G_1}^K(f_1, \dots, f_{n_1}), \Lambda_{G_2}^K(f_{n_1+1}, \dots, f_{n_1+n_2})$$ is universally $L^p$-improving and the other satisfies a non-trivial estimate at the common vertex. Then $\Lambda_G^K(f_1, \dots, f_{n_1+n_2-1})$ is $L^p$-improving. \end{theorem} 

\vskip.125in 

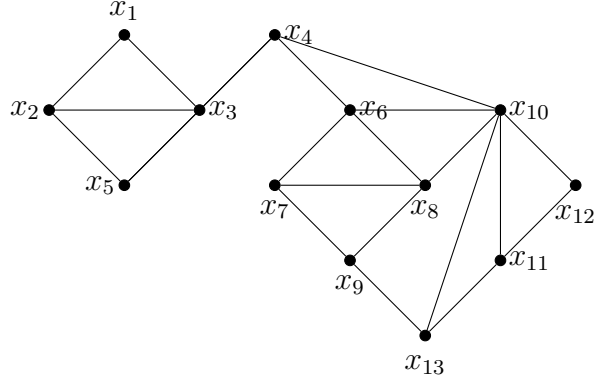
\begin{figure}[H]
\begin{tikzpicture}[node distance=2cm]
  \tikzstyle{every node}=[circle, draw, fill=black, inner sep=0pt, minimum width=4pt]

  \node[label=above:$x_1$] (x1) at (0,0) {};
  \node[label=left:$x_2$] (x2) at (-1,-1) {};
  \node[label=right:$x_3$] (x3) at (1,-1) {};
  \node[label=right:$x_4$] (x4) at (2,0) {}; 
  \node[label=left:$x_5$] (x5) at (0,-2) {};

  \draw (x1) -- (x2);
  \draw (x1) -- (x3);
  \draw (x2) -- (x5);
  \draw (x2) -- (x3);
  \draw (x3) -- (x4);
  \draw (x3) -- (x5);
  \draw (x5) -- (x4);

  \node[label=right:$x_6$] (x6) at (3,-1) {};
  \node[label=below:$x_7$] (x7) at (2,-2) {};
  \node[label=below:$x_8$] (x8) at (4,-2) {};
  \node[label=below:$x_9$] (x9) at (3,-3) {};
  \node[label=right:$x_{10}$] (x10) at (5,-1) {};
  \node[label=right:$x_{11}$] (x11) at (5,-3) {};
  \node[label=below:$x_{12}$] (x12) at (6,-2) {};
  \node[label=below:$x_{13}$] (x13) at (4,-4) {};

  \draw (x4) -- (x6);
  \draw (x6) -- (x7);
  \draw (x6) -- (x8);
  \draw (x7) -- (x9);
  \draw (x8) -- (x9);
  \draw (x7) -- (x8);
  \draw (x10) -- (x11);
  \draw (x10) -- (x12);
  \draw (x11) -- (x12);
  \draw (x13) -- (x11);
  \draw (x13) -- (x10);
  \draw (x9) -- (x13);
  \draw (x8) -- (x10);
  \draw (x6) -- (x10);
  \draw (x4) -- (x10);

\end{tikzpicture}
\caption{Two connected graphs with a common vertex $x_4$} 
\label{twographs}
\end{figure}

This can be easily extended to a broader class of graphs with more graphs being connected.

\begin{definition}
Let $G$ be a connected graph with sub-graphs $G_1,\ldots,G_k$. We say $G$ is a tree graph formed by $G_1,\ldots,G_k$ if for each $1\leq i < j \leq k$ we have that $G_i$ and $G_j$ share at most one vertex and the graph formed by assigning a vertex $v_i$ to each $G_i$ and drawing an edge between $v_i$ and $v_j$ if $G_i$ and $G_j$ share exactly one vertex forms a graph that is a tree.
\end{definition}

An induction argument using Theorem \ref{joininggraphs} repeatedly yields the following Corollary.

\begin{corollary}
Suppose that a connected graph $G$ is a tree graph formed by $G_1,\ldots,G_k$ and $K$ is a locally integrable kernel. Then $\Lambda_G^K$ is $L^p$ improving if at least one of the $\Lambda_{G_{i}}^K$ is universally $L^p$-improving and the remaining ones satisfy non-trivial estimates at all vertices that connect to another graph.
\end{corollary}

\vskip.125in 

\subsection{Distance graph results}

Our next result says that if $K=\sigma$, the arc-length measure on the unit circle $S^1$, then $\Lambda_G^K$ is $L^p$-improving provided that the underlying multi-linear form is, in a suitable sense, well-defined. We shall need the following definition which generalizes the notion of a locally infinitesimally rigid configuration in \cite{CIMP21}. We shall state a more general version of this definition in Section \ref{future}. 

\begin{definition}\label{def: regularly realizable} Let $G$ be a graph with vertices $\{1,\ldots, n\}$ and edge set $E$. Let $F : \mathbb R^{2n} \to \mathbb R^{|E|}$ be the map
$$
    F(x_1, \ldots, x_n) = (|x_i - x_j|)_{ij} \qquad \text{ for } \{i,j\} \in E, \ i < j,
$$
and set $M = F^{-1}(\mathbf 1)$ where $\mathbf 1 \in \mathbb R^{|E|}$ is the vector with $1$ in every entry. We say that $G$ is \emph{regularly realizable} in $\mathbb R^2$ if $M$ is nonempty and $\mathbf 1$ is a regular value of $F$.
\end{definition}

We discuss some examples of graphs that are regularly realizable and graphs that are not regularly realizable in Section \ref{sec: reg realizable}.

\begin{remark} An interesting question we were not able to resolve is whether Definition \ref{def: regularly realizable} is equivalent to the statement that the corresponding multi-linear kernel (\ref{localintegrability}) is locally integrable. We hope to address this issue in the sequel. For example, the Leray measure for the 6-cycle is singular, but it is integrable. \end{remark} 


\vskip.125in 

\begin{theorem}\label{allgraphsgood}
    Let $G$ be a connected graph on $n \geq 2$ vertices which is regularly realizable in $\mathbb R^2$. Then, the multilinear form $\Lambda_G$ is bounded on $L^{p_1}(\mathbb R^2) \times \cdots \times L^{p_n}(\mathbb R^2)$ for all $(\frac{1}{p_1}, \ldots, \frac{1}{p_n})$ contained in the convex hull of the points
    $$
        \{e_1, \ldots, e_n\} \cup \left\{\frac 2 3 e_i + \frac 2 3 e_j \right\}_{\{i,j\} \in E}.
    $$
    In particular, $\Lambda_G$ is uniformly $L^p$-improving.
\end{theorem}

\begin{remark} With a bit of work, Theorem \ref{allgraphsgood} can be extended to higher dimensions. A much more general formulation should be possible as well. We shall address this issue in the sequel. See Section \ref{future} for more details. \end{remark} 


\vskip.125in 

We also conducted a couple of case studies where, at least in some cases, sharp exponents can be obtained for the problems described above. 

\vskip.125in 

\begin{theorem} \label{thm:triangle}
Let $G$ denote a complete graph on three vertices.
The trilinear form $\Lambda_G$ is bounded by $L^{p_1}(\mathbb R^2) \times L^{p_2}(\mathbb R^2) \times L^{p_3}(\mathbb R^2)$ for all $(\frac{1}{p_1}, \frac{1}{p_2}, \frac{1}{p_3})$ contained in the polygon with the vertices:  
$$(1,0,0), \  (0,1,0), \ (0, 0, 1), \ \left(\frac{2}{3}, \frac{2}{3}, 0 \right), \ \left(\frac{2}{3}, 0, \frac{2}{3} \right), \ \left(0, \frac{2}{3}, \frac{2}{3}\right), \ \left(\frac{1}{2}, \frac{1}{2}, \frac{1}{2}\right).$$

The exponents are optimal modulo the endpoint $ \ \left(\frac{1}{2}, \frac{1}{2}, \frac{1}{2}\right)$ where a restricted strong type estimate holds.
\end{theorem}

\vskip.125in 

\begin{figure}[H]
\begin{tikzpicture}[node distance=2cm]

  \tikzstyle{every node}=[circle, draw, fill=black, inner sep=0pt, minimum width=4pt]

  \node[label=left:$x_1$] (x1) at (0,0) {};
  \node[label=left:$x_2$] (x2) at (-0.5,-1) {};
  \node[label=right:$x_3$] (x3) at (0.5,-1) {};

  \draw (x1) -- (x2);
  \draw (x1) -- (x3);
  \draw (x2) -- (x3);

\end{tikzpicture}
\caption{Complete graph on three vertices} 
\label{triangle}
\end{figure}
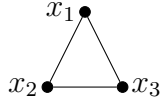 

\vskip.125in 

\begin{theorem} \label{thm:2-chain}
Let $G$ denote a chain on three vertices and let $\Lambda_G$ denote the corresponding trilinear form. More precisely, 
\begin{equation} \label{eq:2-chain}
\Lambda_G(f_1,f_2,f_3) := \int f*\sigma(x) g*\sigma(x) h(x)dx=\int Af(x)Ag(x)h(x)dx,
\end{equation}
where $A$ is the circular averaging operator. 

\vskip.125in 

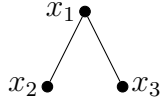
\begin{figure}[H]
\begin{tikzpicture}[node distance=2cm]
\label{hinge}

  \tikzstyle{every node}=[circle, draw, fill=black, inner sep=0pt, minimum width=4pt]

  \node[label=left:$x_1$] (x1) at (0,0) {};
  \node[label=left:$x_2$] (x2) at (-0.5,-1) {};
  \node[label=right:$x_3$] (x3) at (0.5,-1) {};

  \draw (x1) -- (x2);
  \draw (x1) -- (x3);

\end{tikzpicture}
\caption{A path on three vertices} 
\end{figure} 

The trilinear form $\Lambda_G$ is bounded by $L^{p_1}(\mathbb R^2) \times L^{p_2}(\mathbb R^2) \times L^{p_3}(\mathbb R^2)$ for all $(\frac{1}{p_1}, \frac{1}{p_2}, \frac{1}{p_3})$ contained in the polygon with the vertices:  
$$(1,0,0), \  (0,1,0), \ (0, 0, 1),  \ \left(\frac{2}{3}, 0, \frac{2}{3}\right), \ \left(0, \frac{2}{3}, \frac{2}{3}\right), \ \left(1, \frac{1}{2}, 0 \right), \ \left(\frac{1}{2},1,0 \right).$$
\end{theorem}

\begin{remark} The two missing endpoints as imposed by the necessary conditions for the boundedness of $\Lambda_G$ defined in \eqref{eq:2-chain} are
$(\frac{1}{2}, \frac{5}{6},\frac{1}{3})$ and $(\frac{5}{6}, \frac{1}{2}, \frac{1}{3})$.

\vskip.125in 

For the endpoints $(1, \frac{1}{2}, 0)$ and  $(\frac{1}{2},1,0)$, we prove the Sobolev bounds which are sufficient to interpolate and obtain the desired $L^p$ bounds. 
\end{remark}

\vskip.125in 

\begin{remark} By comparing the exponents in Theorems \ref{thm:triangle} and \ref{thm:2-chain}, we see that the answer to Question 3 is positive in this case since the path on three vertices is a subgraph of the complete graph on three vertices. \end{remark} 


\vskip.25in 
 
\section{Proof of Theorem \ref{maintree}} 
 
\vskip.125in 
 
  The proof is by induction on the size of the tree. 
  For the sake of notational convenience, we are going to prove a stronger statement. 
  Let $x_1$ be the root of a tree $G$. Then 
 $$ \Lambda_G(f_1, \dots, f_n)=\int U(x_1)f_1(x_1) dx_1$$ for some function $U$. Note that $U$ implicitly depends on $f_2, \dots, f_n$. 
 We are going to prove that for any $q\geq 1$, there exists $p_1,\ldots, p_n >1$ such that 
 $$ {||U\cdot f_1||}_{L^q({\mathbb R}^d)} \leq C \prod_{i=1}^n {||f_i||}_{L^{p_i}({\mathbb R}^d)},$$ with 
\begin{equation} \label{goodqindices} \sum_{i=1}^n \frac{1}{p_i}>\frac{1}{q}. \end{equation} 
 
 To initiate the induction process, we first prove the estimate in the case $n=2$. We have 
 $$ \Lambda^K_G(f_1,f_2)=\int \int K(x_1,x_2) f_1(x_1)f_2(x_2) dx_2dx_1=\int T_K f_2(x_1) f_1(x_1) dx_1,$$ 
 so $U(x_1)=T_Kf_2(x_1).$ 
By H\"older's inequality, we have 
 $$ {||T_Kf_2 \cdot f_1||}_{L^q({\mathbb R}^d)} \leq  {||T_Kf_2||}_{L^{r_2}({\mathbb R}^d)} \cdot {||f_1||}_{L^{p_1}({\mathbb R}^d)},$$ 
 where $p_1, r_2 >1$ and $\frac{1}{p_1} + \frac{1}{r_2}=\frac{1}{q}$.
 
 \vskip.125in 
 
 Since $T_K$ is universally $L^p$-improving, there exists $p_2<r_2$ such that 
 $$ {||T_Kf_2||}_{L^{r_2}({\mathbb R}^d)} \leq C' {||f_2||}_{L^{p_2}({\mathbb R}^d)},$$ and we see that 
 $$ {||T_Kf_2 \cdot f_1||}_{L^q({\mathbb R}^d)} \leq C'' {||f_1||}_{L^{p_1}({\mathbb R}^d)} \cdot {||f_2||}_{L^{p_2}({\mathbb R}^d)},$$ with 
 $$ \frac{1}{p_1}+\frac{1}{p_2}>\frac{1}{p_1}+\frac{1}{r_2}=\frac{1}{q},$$ as desired. 
 
 \vskip.125in 
 
We now proceed to the core of the induction argument. Let $x_1$ denote the variable corresponding to the root of the tree. This can be accomplished by relabeling, if necessary. Suppose that $x_1$ is connected to $x_2, x_3, \dots, x_m$, which is again accomplished by relabeling. By assumption, $x_j$s for $2 \leq j \leq m$ are not connected and each of them is the root of the tree that does not intersect the trees rooted at the other $x_j$s. For $2 \leq j \leq m$, let $I_j$ denote the set of the indices associated with the tree rooted at $x_j$. 
Let $F_j(x_j)$ denote the function corresponding to the tree rooted at $x_j$.  

Integrating in the remaining variables, we see that 
 $$ \Lambda_G^K(f_1, \dots, f_n)=\int \dots \int \prod_{j=2}^m F_j(x_j) \cdot \prod_{j=2}^m K(x-x_j) dx_1 \dots dx_m \cdot f_1(x_1)dx_1$$ 
 $$ =\int \prod_{j=2}^m T_K(F_j)(x_1) \cdot f_1(x_1)dx_1.$$ 
 
 By Holder's inequality, one has 
 \begin{equation} \label{decouple:subtree}
 \big\|\prod_{j=2}^m T_K(F_j) \cdot f_1 \big\|_{L^q(\mathbb{R}^d)} \lesssim 
 {||f_1||}_{L^{p_1}({\mathbb R}^d)} \cdot \prod_{j=2}^m {||T_K(F_j)||}_{L^{r_j}({\mathbb R}^d)},
 \end{equation}
 with 
 \begin{equation} \label{Holder:tree}
 \frac{1}{p_1} + \sum_{j=2}^m \frac{1}{r_j} = \frac{1}{q}.
 \end{equation}
 
 By the assumption that $T_K$ is universally $L^p$-improving, for each $2 \leq j \leq m$, there exists $q_j < r_j$ such that \eqref{decouple:subtree} is bounded (up to a constant) by
 $$  {||f_1||}_{L^{p_1}({\mathbb R}^d)} \cdot \prod_{j=2}^m {||F_j||}_{L^{q_j}({\mathbb R}^d)}.$$
For each $2 \leq j \leq m$, one invokes the induction hypothesis so that
 $$\|F_j\|_{L^{q_j}(\mathbb{R}^d)} \lesssim \prod_{i \in I_j}{||f_i||}_{p_{i}}, $$ 
 for some $p_i>1$ and 
 \begin{equation} \label{indices:subtree}
 \sum_{i \in I_j} \frac{1}{p_i} > \frac{1}{q_j}.
 \end{equation}

One notices that by \eqref{indices:subtree} and \eqref{Holder:tree}, 
$$
\frac{1}{p_1} + \sum_{j=2}^m \sum_{i \in I_j} \frac{1}{p_j} > \frac{1}{p_1} + \sum_{j=2}^m \frac{1}{q_j} > \frac{1}{p_1} + \sum_{j=2}^m \frac{1}{r_j}  = \frac{1}{q},
$$
which completes the proof of the theorem. 
 \vskip.125in 

\vskip.25in 

\section{Proof of Theorem \ref{treeremoval}} 

\vskip.125in 

Suppose that $|G|=n$ and $|G'|=n'$. Performing the contraction operation in the context of $\Lambda_{G'}^K$ leads to 
$$ \Lambda^K_G(F_1, \dots, F_n) \leq C \prod_{j=1}^n {||F_j||}_{L^{r_j}({\mathbb R}^d)},$$ with 
$$ \sum_{j=1}^n \frac{1}{r_j}>1.$$ 

\vskip.125in 

By Theorem \ref{maintree}, or, more precisely, by its proof, 

$$  {||F_j||}_{L^{r_j}({\mathbb R}^d)} \leq \prod_{i=1}^{k(j)} {||f_i||}_{L^{p_{ji}}({\mathbb R}^d)},$$ where 
$$ \sum_{i=1}^{k(j)} \frac{1}{p_{ji}}>\frac{1}{r_j}.$$ 

\vskip.25in 

\section{Proof of Theorem \ref{joininggraphs}} 

\vskip.125in 

Label the graph $G$ such that the functions and vertices corresponding to the graph $G_1$ are labeled by $1,\ldots,n_1$ and the ones corresponding to the graph $G_2$ are labeled by $n_1, \ldots, n_1+n_2-1$ which makes the vertex labled by $n_1$ the single vertex in common between $G_1$ and $G_2$. Without loss of generality we can assume that $\Lambda_{G_2}^K$ is universally $L^p$ improving. By integrating through the vertices of $G_1$ and $G_2$ other than the common one we can write the form $\Lambda_{G}^K(f_1,\ldots,f_{n_1+n_2-1})$ as
$$ \int T_{n_1}^{G_{1}}(f_1,\ldots,f_{n_1 -1})(x_{n_1})\left(T_{n_1}^{G_{2}}(f_{n_1+1},\ldots,f_{n_1+n_2-1})(x_{n_1})f_{n_1}(x_{n_1}) \right)dx_{n_1} .$$
As $\Lambda_{G_1}^{K}$ has a non-trivial estimate at $n_1$ there exist $1 \leq p_1,\ldots,p_{n_1} \leq \infty$ with $p_{n_1}<\infty$ such that
$$ \Lambda_{G}^K(f_1,\ldots,f_{n_1+n_2-1}) \leq C \|f_1\|_{p_1}\ldots\|f_{n_1-1} \|_{p_{n_1-1}}\| T_{n_1}^{G_{2}}(f_{n_1+1},\ldots,f_{n_1+n_2-1})f_{n_1}\|_{p_{n_1}} $$
where $\frac{1}{p_{1}}+\ldots+\frac{1}{p_{n_1}}\geq 1$. We now use H\"{o}lder and bound
$$ \| T_{n_1}^{G_{2}}(f_{n_1+1},\ldots,f_{n_1+n_2-1})f_{n_1}\|_{p_{n_1}} \leq \| T_{n_1}^{G_{2}}(f_{n_1+1},\ldots,f_{n_1+n_2-1})\|_{q_{n_1}'}\|f_{n_1}\|_{q_{n_1}}$$
where $\frac{1}{q_{n_1}'}+\frac{1}{q_{n_1}}=\frac{1}{p_{n_1}}$ and since $p_{n_1}<\infty$ we can ensure $q_{n_1}'<\infty$. Since $\Lambda_{G_2}^K$ is universally $L^p$ improving there exist $1\leq q_{n_1+1},\ldots,q_{n_1+n_2-1}\leq\infty$ with $\frac{1}{q_{n_1+1}}+\ldots+\frac{1}{q_{n_1+n_2-1}}>\frac{1}{q_{n_1}'}$. Combined this now gives us a bound
$$ \Lambda_{G}^K(f_1,\ldots,f_{n_1+n_2-1}) \leq C \|f_1\|_{p_1}\ldots\|f_{n_1-1} \|_{p_{n_1-1}}\|f_{n_1}\|_{q_{n_1}}\| f_{n_1+1} \|_{q_{n_1+1}}\ldots \| f_{n_1+n_2-1} \|_{q_{n_1+n_2-1}} $$
where
\begin{align*}
\frac{1}{p_1} + \ldots + \frac{1}{p_{n_1 - 1}} + \frac{1}{q_{n_1}}+\frac{1}{q_{n_1+1}}+\ldots+\frac{1}{q_{n_1+n_2-1}} &>\frac{1}{p_1} + \ldots + \frac{1}{p_{n_1 - 1}} + \frac{1}{q_{n_1}}+\frac{1}{q_{n_1}'} \\
&= \frac{1}{p_1} + \ldots + \frac{1}{p_{n_1 - 1}} + \frac{1}{p_n} \\
&\geq 1
\end{align*}
which shows that $\Lambda_{G}^{K}$ is $L^p$ improving.

\vskip.25in 

\section{Case Study on Trilinear forms: proof of theorems \ref{thm:triangle} and \ref{thm:2-chain}} 
 
\vskip.25in 

This section is devoted to the study of the trilinear forms described in Theorem \ref{thm:triangle} and Theorem \ref{thm:2-chain}. 
 
 
\subsection{Complete graph on three vertices} 
We first state the necessary conditions for the boundedness exponents for the corresponding trilinear form. 
 \begin{theorem} 
 Let $G$ denote a complete graph on three vertices. 
 If the trilinear form has the boundedness property 
$$ \Lambda_G(f,g,h) \leq {||f||}_{L^{p_1}(\mathbb R^d)}{||g||}_{L^{p_2}(\mathbb R^d)}{||h||}_{L^{p_3}(\mathbb R^d)},$$
then the boundedness exponents $p_1, p_2, p_3$ satisfy the following conditions:
\begin{enumerate}
\item
$\frac{1}{p_1} + \frac{1}{p_2} + \frac{1}{p_3} \geq 1$;
\item
$\frac{1}{p_1}+ \frac{1}{p_2} + \frac{d}{p_3} \leq d$;
\item
$\frac{1}{p_1} + \frac{d}{p_2} + \frac{1}{p_3} \leq d$;
\item
$\frac{d}{p_1} +\frac{1}{p_2} + \frac{1}{p_3} \leq d$;
\item
$\frac{d+1}{p_1} + \frac{d+1}{p_2} +\frac{2d}{p_3} \leq 3d-1$;
\item
$\frac{d+1}{p_1} + \frac{2d}{p_2} +\frac{d+1}{p_3} \leq 3d-1$;
\item
$\frac{2d}{p_1} + \frac{d+1}{p_2} +\frac{d+1}{p_3} \leq 3d-1$.
\end{enumerate}
  \end{theorem}
 
 \begin{remark} 
In dimension 2, the boundedness exponent $(\frac{1}{p_1}, \frac{1}{p_2}, \frac{1}{p_3})$ is contained in the polygon with the vertices:  $$(1,0,0), \  (0,1,0), \ (0, 0, 1), \ (\frac{2}{3}, \frac{2}{3}, 0), \ (\frac{2}{3}, 0, \frac{2}{3}), \ (0, \frac{2}{3}, \frac{2}{3}), \ (\frac{1}{2}, \frac{1}{2}, \frac{1}{2}),$$ 
which agree with the exponents described in Theorem \ref{thm:triangle} and verify the optimality claimed in the theorem.
\end{remark}

To prove the desired boundedness exponents in Theorem \ref{thm:triangle}, one observes that the trilinear form can be written as 
 $$
 \int f(x) B_{\frac{\pi}{3}} (g,h) (x) dx + \int f(x) B_{-\frac{\pi}{3}} (g,h) (x) dx,
 $$
 where the bilinear operators are called the generalized Radon transform in \cite{GIKL17} defined by
 \begin{equation} \label{eq:bilinear-Radon}
 B_\theta (g,h) (x) := \int g(x-y) h(x-\Theta y) d\sigma y,
 \end{equation}
 with $\Theta$ is the rotation by the angle $\theta$. 
All the endpoint estimates can be derived from the corresponding bounds for the bilinear operators established in \cite{GIKL17}, from which one can interpolate to obtain all the boundedness exponents. 

 \subsection{A chain on three vertices} 
We start with the necessity of the boundedness property for the trilinear form corresponding to the chain on three vertices. 
\begin{theorem}
If the trilinear form defined in \eqref{eq:2-chain} has the boundedness property 
$$ \Lambda_G(f,g,h) \leq {||f||}_{L^{p_1}(\mathbb R^d)}{||g||}_{L^{p_2}(\mathbb R^d)}{||h||}_{L^{p_3}(\mathbb R^d)},$$
then the boundedness exponents $p_1, p_2, p_3$ satisfy the following conditions:
\begin{enumerate}
\item
$\frac{1}{p_1} + \frac{1}{p_2} + \frac{1}{p_3} \geq 1$;
\item
$\frac{1}{p_1}+ \frac{1}{p_2} + \frac{d}{p_3} \leq d$;
\item
$\frac{d}{p_1} + \frac{d}{p_2} + \frac{1}{p_3} \leq 2d-1$;
\item
$\frac{d}{p_1} + \frac{1}{p_3} \leq d$;
\item
$\frac{d}{p_2} +\frac{1}{p_3} \leq d$. 
\end{enumerate}
\end{theorem}

\begin{remark}
In dimension $2$, the boundedness exponent $(\frac{1}{p_1}, \frac{1}{p_2}, \frac{1}{p_3})$ is contained in the polygon with the vertices: \\
$$(1,0,0), \ (0,1,0), \ (0,0,1), \ (\frac{2}{3}, 0, \frac{2}{3}), \ (0, \frac{2}{3}, \frac{2}{3}), \ (\frac{1}{2},1,0), \ (1, \frac{1}{2}, 0), \ (\frac{1}{2}, \frac{5}{6}, \frac{1}{3}), \ (\frac{5}{6}, \frac{1}{2}, \frac{1}{3}).$$
\end{remark}

\begin{proof}
\begin{enumerate}
\item
Let $f(x) = g(x) = h(x) = \chi_{B_R}(x)$, where $B_R$ denotes the ball of sufficiently large radius $R$ centered at the origin. One observes that
\begin{equation} \label{pointwise:big-ball}
Af(x) = Ag(x) \gtrsim \chi_{B_R}(x). 
\end{equation}
Applying the pointwise estimate \eqref{pointwise:big-ball}, the trilinear form is bounded below by
\begin{equation*}
  \Lambda_G(f,g,h) = \int Af(x) Ag(x) h(x) dx \gtrsim |B_R| \gtrsim R^{d}. 
\end{equation*}
Meanwhile, 
$$
\|f\|_{p_1} \approx R^{\frac{d}{p_1}},  \ {||g||}_{p_1} \approx R^{\frac{d}{p_2}},  \ {||h||}_{p_3} \approx R^{\frac{d}{p_3}}.
$$
Thus the boundedness property implies 
\begin{equation} \label{ineq:exp:largeball}
R^d \lesssim R^{\frac{d}{p_1} + \frac{d}{p_2}+\frac{d}{p_3}}.
\end{equation}
In the case $R$ is sufficiently large, \eqref{ineq:exp:largeball} yields
$$
\frac{1}{p_1} + \frac{1}{p_2} +\frac{1}{p_3} \geq 1.
$$
\item
Let $f(x) = g(x) = \chi_{A_\delta}(x)$, where $A_\delta$ is the annulus of radius $1$ and thickness $\delta$ centered at the origin. Let $h(x) = \chi_{B_\delta}(x)$, where $B_\delta$ is the ball of radius $\delta$ (small) centered at the origin. 
We shall note that
\begin{equation*}
Af(x) = Ag(x) \gtrsim \chi_{B_\delta}(x),
\end{equation*} 
which gives 
$$
\Lambda_G(f,g,h) = \int Af(x) Ag(x) h(x) dx \gtrsim  |B_\delta| = \delta^d.
$$
Furthermore, one has 
$$
{||f||}_{p_1} \approx \delta^{\frac{1}{p_1}}, \ {||g||}_{p_2} \approx \delta^{\frac{1}{p_2}}, \ {||h||}_{p_3} \approx \delta^{\frac{d}{p_3}}.
$$
Due to the boundedness property 
\begin{equation*}
\delta^{d} \lesssim \delta^{\frac{1}{p_1} + \frac{1}{p_2} + \frac{d}{p_3}},
\end{equation*}
and the fact that $\delta \ll 1$, one concludes that
$$
\frac{1}{p_1} + \frac{1}{p_2} + \frac{d}{p_3} \leq d.
$$
\item
Let $f(x)=g(x)=\chi_{B_{\delta}}(x)$ and $h(x) = \chi_{A_{\delta}}(x)$. 
We have 
$$Af(x)=Ag(x) \gtrsim \delta^{d-1} \cdot \chi_{A_{\delta}}(x).$$ 
It follows that 
$$\Lambda_G(f,g,h) = \int Af(x) Ag(x) h(x) dx \gtrsim \delta^{2(d-1)} |A_\delta| = \delta^{2d-1},$$ 
whereas 
$$ {||f||}_{p_1} \approx \delta^{\frac{d}{p_1}}, \ {||g||}_{p_2} \approx \delta^{\frac{d}{p_2}}, \ {||h||}_{p_3} \approx \delta^{\frac{1}{p_3}},$$
which leads to the condition 
\begin{equation} \label{restr1} 
\frac{d}{p_1}+\frac{d}{p_2} + \frac{1}{p_3} \leq 2d-1. 
\end{equation} 
\item
Let $f(x) = \chi_{B_\delta}(x)$, $g(x) \equiv 1$ and $h(x) = \chi_{A_\delta}(x)$. One can easily verify that $Ag(x) \equiv 1$ and thus
$$
\Lambda_G(f,g,h) = \int Af(x) Ag(x) h(x) dx \gtrsim \delta^{d-1} |A_\delta| \gtrsim \delta^d.$$ 
On the other hand, one has
$$
 {||f||}_{p_1} \approx \delta^{\frac{d}{p_1}}, \ {||g||}_{\infty}=1, \ {||h||}_{p_3} \approx \delta^{\frac{1}{p_3}},
 $$
 which gives rise to the inequality 
 $$
 \frac{d}{p_1} + \frac{1}{p_3} \leq d. 
 $$
\item
Let $f(x) \equiv 1$, $g(x) = \chi_{B_\delta}(x)$ and $h(x)= \chi_{A_\delta}(x)$. The same reasoning for $(4)$ applies to this case. 
\end{enumerate}

\end{proof}
\vskip.125in 

We will now focus on the sufficient conditions for the boundedness exponents described in Theorem \ref{thm:2-chain}. The boundedness property which will be discussed in the following two theorems indeed holds for the input functions defined on $\mathbb R^d$ whereas Theorem \ref{thm:2-chain} concerns the case $d=2$.

The first theorem establishes boundedness of the trilinear form at the endpoints $(1,0,0)$,  $(0,1,0)$, $(0,0,1)$, $(\frac{d}{d+1}, 0, \frac{d}{d+1})$, $(0, \frac{d}{d+1}, \frac{d}{d+1})$.

\begin{theorem} \label{2chain}  Let $\Lambda_G$ denote the trilinear form defined in \eqref{eq:2-chain}. Then for every $(r_1,r_2,p_3)$ satisfying $r_1, r_2, p_3 \ge 1$, $\frac{1}{r_1}+\frac{1}{r_2}+\frac{1}{p_3}=1$, $\left(\frac{1}{p_1}, \frac{1}{r_1} \right)$ and $\left(\frac{1}{p_2}, \frac{1}{r_2} \right)$ are in ${\mathcal T}$, one has
 $$ \Lambda_G(f,g,h) \leq {||f||}_{p_1}{||g||}_{p_2}{||h||}_{p_3}.$$  
 \end{theorem} 
 
 \vskip.125in 
 
 \begin{proof}
 By H\"{o}lder, 
 
 $$  \Lambda_G(f,g,h) \leq {||Af||}_{r_1} \cdot {||Ag||}_{r_2} \cdot {||h||}_{p_3},$$ with $r_1,r_2 \ge 1$, $\frac{1}{r_1}+\frac{1}{r_2}+\frac{1}{p_3}=1$. The right hand side is bounded by a constant multiple of 
 
 $$ {||f||}_{p_1} \cdot {||g||}_{p_2} \cdot {||h||}_{p_3},$$ where $\left(\frac{1}{p_1}, \frac{1}{r_1} \right)$ and $\left(\frac{1}{p_2}, \frac{1}{r_2} \right)$ are contained in ${\mathcal T}$, the triangle with the endpoints $(0,0)$, $(1,1)$, and 
 $\left( \frac{d}{d+1}, \frac{1}{d+1} \right)$.

\end{proof}

 \vskip.125in 
 
The second theorem concerns the Sobolev bounds for the trilinear form at the endpoints $(\frac{d-1}{d}, 1,0)$ and $(1, \frac{d-1}{d},0)$.

 \begin{theorem} \label{endpoint:double-convolution}
 Let $\Lambda_G$ denote the trilinear form defined in \eqref{eq:2-chain}. Then for any $\epsilon >0$, 
 $$
 \Lambda_G(f,g,h) \lesssim \|f\|_{H^{\epsilon,\frac{d}{d-1}}} \|g\|_{1} \|h\|_{\infty}.
 $$
 Similarly, $
 \Lambda_G(f,g,h) \lesssim \|f\|_{1}\|g\|_{H^{\epsilon,\frac{d}{d-1}}}  \|h\|_{\infty}.
 $
 \end{theorem}
 
 \begin{proof}
 Due to the symmetry between $f_1$ and $f_2$, we will focus on the proof of the first estimate. We assume that all the input functions are non-negative so that
 \begin{align*}
 \Lambda_G(f,g,h) = \int Af(x) Ag(x) h(x) dx \leq \|h\|_\infty \int Af(x)Ag(x)dx.
 \end{align*}
The integral on the right hand side can be explicitly written as
\begin{align*}
\int f(x-u) g(x-v) d\sigma(u)d\sigma(v) = \int f(z-u+v)g(z)dzd\sigma(u)d\sigma(v),
\end{align*}
which can be estimated by H\"older as follows:
$$
 \int f(z-u+v)g(z)dzd\sigma(u)d\sigma(v) \leq \|g\|_1 \bigg\| \int f(\cdot-u+v) d\sigma(u) d\sigma(v)\bigg\|_\infty.
$$
Now one can apply the Fourier inversion formula to deduce that
$$
\bigg\| \int f(\cdot-u+v) d\sigma(u) d\sigma(v)\bigg\|_\infty \leq \bigg\|\mathcal{F}\left(\int f(\cdot-u+v) d\sigma(u) d\sigma(v) \right) \bigg\|_1.
$$
We recall that 
$$
\mathcal{F}\left(\int f(\cdot-u+v) d\sigma(u) d\sigma(v) \right)(\xi) = \hat{f}(\xi) \hat{\sigma}(\xi)\check{\sigma}(\xi),
$$
where $\hat{\sigma}(\xi)$ and $\check{\sigma}(\xi)$ denote the Fourier transform and the inverse Fourier transform of the spherical measure and satisfy the decay $\<\xi\>^{-\frac{d-1}{2}}$. Thanks to the decay property, one has
\begin{align*}
&\bigg\|\mathcal{F}\left(\int f(\cdot-u+v) d\sigma(u) d\sigma(v) \right) \bigg\|_1 \leq \int |\hat{f}(\xi)| |\hat{\sigma}(\xi)||\check{\sigma}(\xi)| d\xi \lesssim \int \<\xi\>^{-(d-1)} |\hat{f}(\xi)| d\xi \\
& \lesssim \left(\int \<\xi\>^{-d} \<\xi\>^{-\epsilon\cdot\frac{d}{d-1}} d\xi \right)^{\frac{d-1}{d}} \left(\int \<\xi\>^{\epsilon\cdot d}|\hat{f}(\xi)|^d d\xi\right)^{\frac{1}{d}} \lesssim \|f\|_{H^{\epsilon, \frac{d}{d-1}}},
\end{align*}
where the last inequality follows from the Hausdorff-Young inequality. This completes the proof of Theorem \ref{endpoint:double-convolution}.
 \end{proof}
 
\vskip.125in

\vskip.25in 

\section{Regularly Realizable Graphs and Proof of Theorem \ref{allgraphsgood}}

\vskip.125in 

We start by describing graphs that are not regularly realizable. 

\vskip.125in 

\subsection{Examples of (not) regularly realizable graphs} \label{sec: reg realizable}

\vskip.125in 

Let $G$ be a graph with vertex set $\{1, \ldots, n\}$ and edge set $E$. A \emph{realization} of the graph $G$ in $\mathbb R^2$ is a list of points $(x_1, \ldots, x_n) \in \mathbb R^{2n}$ (with $x_i \in \mathbb R^2$ for each $i$) such that
\begin{equation}\label{eq: constraints}
    |x_i - x_j| = 1 \qquad \text{ whenever } \qquad \{i,j\} \in E.
\end{equation}


Definition \ref{def: regularly realizable} ensures the configuration space of points in $\mathbb R^{2n}$ satisfying the constraints in \eqref{eq: constraints} is a smooth, embedded submanifold of $\mathbb R^{2n}$ of dimension $2n - |E|$. 
\vskip.125in 

It is not always easy to check that a graph is regularly realizable in $\mathbb R^2$, but some examples are relatively straightforward. For example, if $G$ is a complete graph on three vertices in ${\mathbb R}^2$, the regular realizability condition is almost automatic. See also \cite{CIMP21} where similar calculations are made in the context of locally rigid configurations. Trees are also regularly realizable, but we have Theorem \ref{maintree} to deal with that case.



\begin{example}
    Let $G$ be the $4$-cycle with vertices $\{1,2,3,4\}$ and edge set
    $$
        E = \{\{1,2\},\{2,3\},\{3,4\},\{4,1\}\}.
    $$
    $G$ fails to satisfy condition (2) of Definition \ref{def: regularly realizable}. Note
    $$
        F(x_1,x_2,x_3,x_4) = (|x_1 - x_2|, |x_2 - x_3|, |x_3 - x_4|, |x_4 - x_1|).
    $$
    Its differential at a point in $M$ is a $4 \times 8$ matrix given by
    $$
        dF(x_1, x_2, x_3, x_4) = \begin{bmatrix}
            x_1 - x_2 & - (x_1 - x_2) & 0 & 0 \\
            0 & x_2 - x_3 & - (x_2 - x_3) & 0 \\
            0 & 0 & x_3 - x_4 & -(x_4 - x_1) \\
            -(x_4 - x_1) & 0 & 0 & x_4 - x_1
        \end{bmatrix},
    $$
    where each entry here is a $1 \times 2$ block. The model problematic point is
    $$
        (x_1,x_2,x_3,x_4) = ((0,0),(1,0),(2,0),(1,0)).
    $$
    At this point, the differential reads
    $$
        \begin{bmatrix}
            -e_1 & e_1 & 0 & 0 \\
            0 & -e_1 & e_1 & 0 \\
            0 & 0 & e_1 & -e_1 \\
            -e_1 & 0 & 0 & e_1
        \end{bmatrix},
    $$
    which has rank $3$. 
\end{example}

\subsection{A brief review of Leray measures}

We now move towards the proof of Theorem \ref{allgraphsgood}. We will need the formalism of Leray densities, which are natural measures on the fibers of smooth maps. In fact, as we will show over the section, the Schwartz kernel of $\Lambda_G$ is precisely the Leray measure on $M = F^{-1}(1)$ of Definition \ref{def: regularly realizable}. In order for the Leray measure to exist and be smooth, $M$ must contain no critical points of $F$. Hence, it is natural to assume $G$ be regularly realizable in Theorem \ref{allgraphsgood}.

Let $U$ be an open subset of $\mathbb R^n$ and $F : U \to \mathbb R^m$ be a smooth map for which $y_0 \in \mathbb R^m$ is a regular value of $F$ so that $M = F^{-1}(y_0)$ is a smooth embedded submanifold of $U$. In what follows, we will assume for simplicity's sake that $y_0 = 0$, though this assumption can certainly be removed.

There is a natural way that $F$ induces a measure on $M$. Take an approximation of the identity $\phi_\epsilon$ at $0$ as $\epsilon \to 0$, and consider the weak limit
$$
    \mu = \lim_{\epsilon \to 0} \phi_\epsilon \circ F.
$$
Indeed, this limit exists and does not depend on our choice of $\phi_\epsilon$. This is verified by direct calculation, as follows. Let $f$ be a continuous function on $U$ with compact support, and consider the limit
$$
    \lim_{\epsilon \to 0} \int_{U} f(x) \phi_\epsilon(F(x)) \, dx.
$$
Possibly by shrinking $U$, the components of $y = F(x)$ can be completed to coordinates of $U$ with the inclusion of some new coordinate functions $z_{m+1}, \ldots, z_n$. The change of coordinates $x = \Phi(y,z)$ allows us to directly calculate
\begin{multline*}
    \lim_{\epsilon \to 0} \iint_{\Phi^{-1}(U)} f(\Phi(y,z)) \phi_\epsilon(y) |\det d\Phi(y,z)| \, dy \, dz \\
    = \int_{\mathbb R^{n-m}} f(\Phi(0,z)) |\det d\Phi(0,z)| \, dz.
\end{multline*}
Indeed, $z \mapsto \Phi(0,z)$ is a local parametrization for $M$, and the integral can be written with respect to the volume density 
$$
    |\det d\Phi(0,z)| \, dz = \left| \frac{\partial(F,z)}{\partial x} \right|^{-1} \, dz
$$
This is the \emph{Leray density} on $M$, and it does not depend on our choice of complementary coordinates $z$.

Be warned, the Leray density on $M$ and the density induced by the restriction of the Euclidean metric to $M$ typically do not coincide, but each is always a smooth multiple of the other.



Again suppose $U \subset \mathbb R^n$ is open and $F : U \to \mathbb R^m$ is a smooth map with $0 \in \mathbb R^m$ as a regular value. Furthermore, $\psi : U \to U$ is a volume-preserving diffeomorphism with
$$
    F(\psi(x)) = F(x) \qquad \text{ for all } x \in U.
$$
Then, the restriction $\psi|_M : M \to M$ is also a diffeomorphism. We claim that the Leray density on $M$ is invariant under $\psi$. That is, if $dx'$ denotes the Leray density on $M$, then
$$
    \int_M f(\psi(x')) \, dx' = \int_M f(x') \, dx' \qquad \text{ for all continuous } f.
$$
This follows easily from the limit definition
$$
    \int_M f(x') \, dx' = \lim_{\epsilon \to 0} \int_U f(x) \phi_\epsilon(F(x)) \, dx
$$
of the Leray density.


\subsection{The Leray Density for Regularly Realizable Graphs}

Consider the approximation
$$
    \phi_\epsilon(y) = (2\epsilon)^{-|E|} \prod_{\substack{\{i,j\} \in E \\ i < j}} \chi_{[-1,1]}(\epsilon^{-1}(y_{ij} - 1))
$$
to the point-mass measure at $\mathbf 1$ in $\mathbb R^{|E|}$, and note
$$
    \phi_\epsilon(F(x_1,\ldots,x_n)) = (2\epsilon)^{-|E|} \prod_{\substack{\{i,j\} \in E \\ i < j}} \chi_{[-1,1]}(\epsilon^{-1}(|x_i - x_j| - 1)).
$$
On one hand, its weak limit as $\epsilon \to 0$ is the product
$$
    \prod_{\substack{\{i,j\} \in E \\ i < j}} \sigma(x_i - x_j),
$$
where $\sigma$ is the measure on the unit circle in $\mathbb R^2$. This is precisely the kernel of the multilinear form $\Lambda_G$. On the other hand, if $G$ is regularly realizable, this measure is also the Leray density on $M = F^{-1}(\mathbf 1)$. In particular, the multilinear form $\Lambda_G$ is well-defined and has a locally integrable kernel whenever $G$ is regularly realizable.

\begin{remark}\label{rem: brascamp-lieb} The discussion thus far allows us to write
$$
    \Lambda_G(f_1, \ldots, f_n) = \int_M \prod_{i = 1}^n f_i(\pi_i(x)) \, dx,
$$
where $dx$ here is the Leray density on $M$ and where $\pi_i : (x_1, \ldots, x_n) = x_i$ is the projection $M \to \mathbb R^2$ for each $i$. This setup bares striking resemblance to the setup for the Brascamp-Lieb inequality (see \cite{BrascampLieb} and the references therein). The classical Brascamp-Lieb inequality, however, only applies in the case where $M$ is a linear space. However, there have been some recent extensions of the Brascamp-Lieb inequality over smooth manifolds. The broadest and deepest result so far seems to be \cite{nonlinearBL}. The reader may also be interested in Lehec's work \cite{lehec} in which he presents a method for proving such multilinear inequalities with stochastic methods.
\end{remark}

We will use Fubini's theorem to integrate over the rigid motions of the plane, and then apply Strichartz's bounds for vertices connected by edges. Though simple in principle, the proof requires us to fill in a number of technical details.

Let $\mathcal R(\mathbb R^2)$ be the group of rigid motions in $\mathbb R^2$. This is a three-dimensional connected Lie group. It acts on $M$ via
$$
    \psi(x_1, \ldots, x_n) = (\psi(x_1), \ldots, \psi(x_n)).
$$
We also consider the subset
$$
    M_0 = \{(x_1, \ldots, x_n) \in M : x_1 = 0, \ x_2 = e_1 \}
$$
of configurations with the first two vertices $x_1$ and $x_2$ fixed at the origin and $e_1 = (1,0)$, respectively.

\begin{lemma}
    Assume the hypotheses of Theorem \ref{allgraphsgood} and let $M = F^{-1}(\mathbf 1)$ as above, where $\mathbf 1$ is a regular value of $F$. Then, $M_0$ is a compact, codimension-3, embedded submanifold of $M$, and the map
    $$
        (\psi, x) \mapsto \psi(x)
    $$
    is a diffeomorphism $\mathcal R(\mathbb R^2) \times M_0 \to M$.
\end{lemma}

\begin{proof}
    First, $M_0$ is compact because it is closed and bounded. To see that $M_0$ is bounded, we use the triangle inequality and the assumption that $G$ is connected.
    
    Before proceeding, we set down some notation. Given a point $x = (x_1,\ldots,x_n) \in M$, we write $x_i = (x_i^1, x_i^2)$ for each $i = 1, \ldots, n$. We will also consider the open set $U \subset M$ given by $U = \{x : x_2^2 > 0\}$.
    
    To see $M_0$ is an embedded submanifold of the appropriate dimension, we consider the map $H : U \to \mathbb R^3$ given by
    $$
        H(x) = (x_1^1, x_1^2, x_2^2)
    $$
    where $x = (x_1, x_2, \ldots, x_n)$ and $x_i = (x_i^1, x_i^2)$ for $i = 1, 2$. We claim that $0 \in \mathbb R^3$ is a regular value of $H$, and hence the preimage $M_0 = H^{-1}(0)$ will be a codimension-3 embedded submanifold of $M$. To prove this claim, we consider the rigid motions generated by the following three vector fields on $\mathbb R^2$:
    \begin{align*}
        X(x,y) &= (1,0), \\
        Y(x,y) &= (0,1), \\
        Z(x,y) &= (-y, x).
    \end{align*}
    These vector fields are all elements of the Lie algebra of $\mathcal R(\mathbb R^2)$ and generate the horizontal translations, vertical translations, and rotations about the origin, respectively. Their respective pushforwards through $H$ whenever $(x_1, x_2) = (0,e_1)$ are given by
    \begin{align*}
        dH(X) &= (1,0,0) \\
        dH(Y) &= (0,1,1) \\
        dH(Z) &= (0,0,1),
    \end{align*}
    which span all of $\mathbb R^3$. Hence, the first part of our lemma is proved.

    Next, we show $(\psi, x) \mapsto \psi(x)$ is a diffeomorphism $\mathcal R(\mathbb R^2) \times M_0 \to M$. This map is certainly smooth, so we are left to check that (1) this map is a bijection, and (2) its differential at each point is a linear bijection. For (1) note that if $x \in M$, then there exists a unique rigid motion $\psi_x$ taking $(0,e_1)$ to $(x_1, x_2)$. We then construct an inverse map
    $$
        x \mapsto (\psi_x, \psi_x^{-1}(x))
    $$
    from which (1) is quickly verified.
    
    It suffices to verify (2) at the identity rigid motion $\psi = I$. Fix $x_0 \in M_0$. Consider the differential of our map at $(I, x_0)$,
    $$
        (\partial \psi, \partial x) \mapsto \partial \psi(x_0) + \partial x.
    $$
    To show (2), it suffices to show that this sum above is a direct sum. This however, follows from an identical argument to the one above where we proved $0$ is a regular value of $H$.
    $\qed$
\end{proof}
\bigskip

Let $d\psi$ denote the (left) Haar measure on $\mathcal R(\mathbb R^2)$. That is $d\psi$ is a smooth positive measure on $\mathcal R(\mathbb R^2)$ such that if $\phi$ is another rigid motion, then
$$
    \int_{\mathcal R(\mathbb R^2)} f(\phi \psi) \, d\psi = \int_{\mathcal R(\mathbb R^2)} f(\psi) \, d\psi
$$
for all compactly supported continuous functions $f$ on $\mathcal R(\mathbb R^2)$. Our next lemma ensures there exists a measure on $M_0$ such that the diffeomorphism $\mathcal R(\mathbb R^2) \times M_0 \to M$ in the previous lemma is measure-preserving, where $M$ of course comes equipped with the Leray measure.

\begin{lemma}\label{lem: fubini}
    Let $d\psi$ be the Haar measure on $\mathcal R(\mathbb R^2)$ and let $dx$ denote the Leray measure on $M$. There exists a smooth measure $dx'$ on $M_0$ such that
    $$
        \int_M f(x) \, dx = \int_{\mathcal R(\mathbb R^2)} \int_{M_0} f(\psi(x')) \, dx' \, d\psi
    $$
    for all compactly-supported continuous functions $f$ on $M$.
\end{lemma}

\begin{proof}
    Let $dx'$ be any smooth positive measure on $M_0$. By the previous lemma, there exists a smooth positive function $J$ on $\mathcal R(\mathbb R^2) \times M_0$ for which
    $$
        \int_M f(x) \, dx = \int_{\mathcal R(\mathbb R^2)} \int_{M_0} f(\psi(x')) J(\psi, x') \, dx' \, d\psi.
    $$
    All we must show is that $J$ is constant in the $\psi$-coordinate. Since both our Leray density $dx$ on $M$ and the Haar measure $d\psi$ is invariant under action by a rigid motion $\phi$, we write
    \begin{align*}
        \int_{\mathcal R(\mathbb R^2)} \int_{M_0} f(\psi(x')) J(\psi, x') \, dx' \, d\psi &= \int_M f(x) \, dx \\
        &= \int_M f(\phi(x)) \, dx \\
        &= \int_{\mathcal R(\mathbb R^2)} \int_{M_0} f(\phi \psi(x')) J(\psi, x') \, dx' \, d\psi \\
        &= \int_{\mathcal R(\mathbb R^2)} \int_{M_0} f(\psi(x')) J(\phi^{-1} \psi, x') \, dx' \, d\psi.
    \end{align*}
    $f$ is arbitrary, so $J(\psi, x') = J(x')$ and $J(x') \, dx'$ is our desired measure on $M_0$.
    $\qed$
\end{proof}
\bigskip

We need just one more elementary lemma, which explicitly describes a Haar measure on $\mathcal R(\mathbb R^2)$. In what follows, $T_x \in \mathcal R(\mathbb R^2)$ denotes translation by $x \in \mathbb R^2$, and $R_\theta \in \mathcal R(\mathbb R^2)$ denotes rotation by an angle $\theta$ about the origin.

\begin{lemma}\label{lem: haar}
    Consider the parametrization
    \begin{align*}
        \mathbb R^2 \times S^1 &\to \mathcal R(\mathbb R^2) \\
        (x,\theta) &\mapsto T_x \circ R_\theta.
    \end{align*}
    The measure $dx \, d\theta$ induced by this parametrization is a left Haar measure on $\mathcal R(\mathbb R^2)$.
\end{lemma}

\begin{proof} First, we note that if $(x,\theta)$ and $(x', \theta')$ are elements of $\mathbb R^2 \times S^1$, we have
$$
    (T_x \circ R_\theta) \circ (T_{x'} \circ R_{\theta'}) = T_{x + R_\theta x'} \circ R_{\theta + \theta'}.
$$
Hence, identifying a function $f$ on $\mathcal R(\mathbb R^2)$ with one on $\mathbb R^2 \times S^1$, we must verify
$$
    \int_{\mathbb R^2} \int_{S^1} f(x + R_\theta x', \theta + \theta') \, d\theta' \, dx' = \int_{\mathbb R^2} \int_{S^1} f(x', \theta') \, d\theta' \, dx',
$$
which follows from a change of variables $x' \mapsto R_{-\theta} x'$, followed by $x' \mapsto x' - x$, followed by $\theta' \mapsto \theta' - \theta$.
$\qed$
\end{proof}
\bigskip

With the lemmas in hand, we are ready to prove Theorem \ref{allgraphsgood}.

\bigskip

\begin{proof}
    Suppose without loss of generality that $\{1,2\}$ is in the edge set of $G$. Our aim is to show there exists a constant $C$ for which
    $$
        |\Lambda_G(f_1, \ldots, f_n)| \leq C \|f_1\|_{p_1} \|f_2\|_{p_2} \prod_{k = 3}^n \|f_k\|_{\infty}
    $$
    whenever $(1/p_1, 1/p_2)$ lies within the triangle with vertices $(1,0)$, $(0,1)$, and $(2/3, 2/3)$. We do this by using the change of variables formula of Lemma \ref{lem: fubini} to reduce the standard Strichartz bounds.

    Since the Schwartz kernel of $\Lambda_G$ is a positive measure, we take each of the functions $f_1, \ldots, f_n$ to be positive. We also take each of $f_3, \ldots, f_n$ to be bounded by $1$. We then write (with $dx$ as the Leray density on $M$)
    \begin{align*}
        \Lambda_G(f_1, \ldots, f_n) &= \int_M f_1(x_1) \cdots f_n(x_n) \, dx \\
        &\leq \int_M f_1(x_1) f_2(x_2) \, dx \\
        &= \int_{\mathcal R(\mathbb R^2)} \int_{M_0} f_1(\psi(0)) f_2(\psi(e_1)) \, dx' \, d\psi.
    \end{align*}
    Recall $M_0$ is compact and that the measure $dx'$ from Lemma \ref{lem: fubini} is smooth. Hence, the inner integral evaluates to a constant, say $C$. Continuing on, we have
    \begin{align*}
        &= C \int_{\mathcal R(\mathbb R^2)} f_1(\psi(0)) f_2(\psi(e_1)) \, d\psi.
    \end{align*}
    We now reparametrize the rigid motions by $\mathbb R^2 \times S^1$ as in Lemma \ref{lem: haar} and write
    \begin{align*}
        \int_{\mathcal R(\mathbb R^2)} f_1(\psi(0)) f_2(\psi(e_1)) \, d\psi &= \int_{\mathbb R^2} \int_{S^1} f_1(x) f_2(x + R_\theta e_1) \, d\theta \, dx \\
        &= \langle f_1, \sigma * f_2 \rangle,
    \end{align*}
    where $\sigma$ is the measure supported on the unit circle in $\mathbb R^2$. By Strichartz's bounds (\cite{Str71}), we have
    $$
        \langle f_1, \sigma * f_2 \rangle \lesssim \|f_1\|_{p_1} \|f_2\|_{p_2}
    $$
    with $(1/p_1, 1/p_2)$ lying in the convex hull of $(1,0), (0,1),$ and $(2/3, 2/3)$, as needed.
\end{proof}

\vskip.125in 

\section{Questions, problems and work in progress} 
\label{future} 

\vskip.125in 

In this section, we outline some of the related concepts that we plan to address in the sequel. 

\subsection{Hypergraphs} As we mentioned in the introduction, the problem studied in this paper can be formulated in terms of hypergraphs. More precisely, suppose that $M(f_1, \dots, f_k)$ is a $k$-linear operator with the kernel $K(x^1, \dots, x^k)$. Suppose that $H$ is a connected hypergraph. equipped with the edge map $E$ such that $E(i_1, \dots, i_k)=1$ if the vertices labeled by $i_1, \dots, i_k$ form a hyper-edge. Define $\Lambda_H^K(f_1, \dots, f_n)$ by the expression
$$ \int \dots \int \prod_{\{(i_1,\dots, i_k): E(i_1, \dots, i_k)=1\}} K(x^{i_1}, \dots, x^{i_k})f_1(x^1)f_2(x^2) \dots f_n(x^n) \prod_{j=1}^n dx_j.$$ 

Analogously to the linear case, it is interesting to ask how the structure or the hypergraph $H$ influences whether $\Lambda_H^K(f_1, \dots, f_n)$ is $L^p$-improving, in the same sense as before. For example, suppose that $d=2$ and 
$$ M(f_2,f_3)(x^1)=\lim_{\epsilon \to 0^{+}} \int \int \sigma^{\epsilon}(x^1-x^3)\sigma^{\epsilon}(x^1-x^3)\sigma^{\epsilon}(x^2-x^3)f_2(x^2)f_3(x^3)dx_2dx_3.$$

Let $H$ be the hyper-graph on $5$ vertices where the vertices labeled by $(1,2,3)$ are connected by a hyper-edge, and the vertices $(3,4,5)$ are connected by a hyper-edge. Then if we let 
$$ K^{\epsilon}(x^1,x^2,x^3)=\sigma^{\epsilon}(x^1-x^3)\sigma^{\epsilon}(x^1-x^3)\sigma^{\epsilon}(x^2-x^3),$$ we have 
$$ \Lambda_H^{K,\epsilon}(f_1,f_2,f_3,f_4,f_5)=\int \dots \int K^{\epsilon}(x^1,x^2,x^3)K(x^3,x^4,x^5)\prod_{j=1}^5 f_j(x^j)dx_j.$$

In this case, the right-hand side can be interpreted as $\Lambda_G^K$, where $G$ is a graph with $5$ vertices with edges at $(i,j)=(1,2), (2,3), (1,3), (3,4), (3,5), (4,5)$. 

\vskip.125in 

\begin{figure}[H]
\begin{tikzpicture}[node distance=2cm]

  \tikzstyle{every node}=[circle, draw, fill=black, inner sep=0pt, minimum width=4pt]

  \node[label=left:$x_1$] (x1) at (0,0) {};
  \node[label=left:$x_2$] (x2) at (-0.5,-1) {};
  \node[label=right:$x_3$] (x3) at (0.5,-1) {};

  \draw (x1) -- (x2);
  \draw (x1) -- (x3);
  \draw (x2) -- (x3);
  
  \node[label=right:$x_4$] (x4) at (1.5,-1) {};
  \node[label=right:$x_5$] (x5) at (1,-2) {};

  \draw (x3) -- (x4);
  \draw (x3) -- (x5);
  \draw (x4) -- (x5);

\end{tikzpicture}
\caption{Complete graph on three vertices, with an additional triangle} 
\label{twotriangles}
\end{figure}
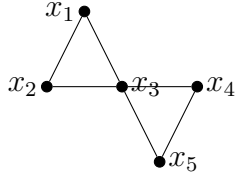

\vskip.125in 

While the example above can be easily expressed in the graph, rather than hypergraph language, the following case we will study in the sequel is very different. Let $\Pi(f,g)(x)$ denote the paraproduct, i.e 
$$ \Pi(f,g)(x)=p.v.\int f(x-t)g(x-s) K_{CZ}(t,s)dtds,$$
where $K_{CZ}$ denotes the Calder\'on-Zygmund kernel. Please see \cite{St-CZ} and \cite{CoifmanMeyer}  for a more detailed description of Caleder\'on-Zygmund operators and their multilinear variant, namely paraproduct, respectively.

Let $K(x^1,x^2,x^3) = K_{CZ}(x^3-x^1,x^3-x^2)$. Then
$$ <\Pi(f_1,f_2),f_3>=\int \int \int K(x^1,x^2,x^3) f_1(x^1)f_2(x^2)f_3(x^3)dx^1dx^2dx^3.$$

Let $H$ be a hypergraph on $4$ vertices where every triple of distinct vertices is connected by a hyper-edge. Then we have 
$$ \Lambda_H^K(f_1,f_2,f_3,f_4)=\int \dots \int \prod_{1 \leq i_1<i_2<i_3 \leq 4} K(x^{i_1}, x^{i_2}, x^{i_3}) \prod_{j=1}^4 f_i(x^j) dx^j.$$

From the work of Coifman and Meyer we know that 
$$ \Pi: L^{p_1}({\mathbb R}) \times L^{p_2}({\mathbb R}) \to L^p({\mathbb R})$$ for $\frac{1}{2} < p < \infty$, $1 < p_1, p_2 \leq \infty$ and $\frac{1}{p_1}+\frac{1}{p_2}=\frac{1}{p}$. Here it is not possible to do better than the Holder exponents, so the question is whether 
$$ \Lambda_H^K(f_1,f_2,f_3,f_4) \leq C \prod_{j=1}^4 {||f_j||}_{L^{p_j}({\mathbb R})}$$ with $\sum_{j=1}^4 \frac{1}{p_j}=1$ and suitable restrictions on the size of $p_i$s. This problem falls into the general framework of singular Brascamp-Lieb inequalities, which has been formulated in \cite{DurcikThiele} and \cite{Muscalu}. Though some cases have been investigated in \cite{DurcikThiele}, \cite{Muscalu} and later works, 
Though a wide class of multilinear singular integral operators (see, for example, \cite{LT97}, \cite{Muscalu07} and \cite{BM22}) have been well understood, the particular question we are interested in the sequel, such as the boundedness of the $4$-linear form
\begin{align*}
&\Lambda_H^K(f_1,f_2,f_3,f_4) = \int K(x^1,x^2,x^3) K(x^1,x^2,x^4) f_1(x^1) f_2(x^2) f_3(x^3) f_4(x^4) dx^1dx^2dx^3dx^4\\
=& \int f_1(x-t_1) f_2(x-s_1) f_3(x) f_4(x-t_1+t_2) K_{CZ}(t_1,s_1) K_{CZ}(t_2,s_1-t_1+t_2) dt_1dt_2ds_1dx,
 \end{align*}
remains open.  

\vskip.125in 

\subsection{Generalized Radon transforms} For this description, a generalized Radon transform is the linear operator of the form $T_K$, where the distribution kernel 
$$ K(x,y)=\int_0^{\infty} e^{2 \pi i \tau  (\phi(x,y)-t)} \psi(x,y) d\tau.$$ Here $\psi$ is a smooth cut-off function, $\phi: {\mathbb R}^d \times {\mathbb R}^d \to {\mathbb R}$ is a continuous function, smooth away from the diagonal, such that 
\begin{equation} \label{minimalphiassumptions} \{x: \phi(x,y)=t\} \ \text{and} \ \{y: \phi(x,y)=t\} \end{equation} are smooth immersed submanifolds of ${\mathbb R}^d$. In the case when $\phi(x,y)=|x-y|$ and $t=1$, we recover the spherical averaging operator 
$$ T_Kf(x)=\int_{S^{d-1}} f(x-y) d\sigma(y),$$ where $\sigma$ is the surface measure on the unit sphere $S^{d-1}$. 

\vskip.125in 

If one assumes, in addition, that 
\begin{equation} det
\begin{pmatrix} 
 0 & \nabla_{x}\phi \\
 -{(\nabla_{y}\phi)}^{T} & \frac{\partial^2 \phi}{dx_i dy_j}
\end{pmatrix}
\neq 0
\end{equation} on the set $\{(x,y): \phi(x,y)=t \}$. then one can use classical methods ($TT^{*}$) to see that 
$$ T_K: L^p({\mathbb R}^d) \to L^q({\mathbb R}^d) $$ for $(\frac{1}{p}, \frac{1}{q})$ in a triangle with the endpoints $(0,0)$, $(1,1)$, and $(\frac{d}{d+1}, \frac{1}{d+1})$. See, for example, \cite{St93}, \cite{So93} and the references contained therein. 

\vskip.125in 

We shall consider the following generalization of Definition \ref{def: regularly realizable}. 

\begin{definition}\label{def: general regularly realizable} Let $G$ be a graph with vertices $\{1,\ldots, n\}$ and edge set $E$. Let $F : \mathbb R^{2n} \to \mathbb R^{|E|}$ be the map
$$
    F(x_1, \ldots, x_n) = (\phi(x^i,x^j))_{ij} \qquad \text{ for } \{i,j\} \in E, \ i < j,
$$
where $\phi$ is as in (\ref{minimalphiassumptions}) above, and set $M = F^{-1}(\mathbf 1)$ where $\mathbf 1 \in \mathbb R^{|E|}$ is the vector with $1$ in every entry. We say that $G$ is \emph{regularly realizable} in $\mathbb R^2$ if $M$ is nonempty and $\mathbf 1$ is a regular value of $F$.
\end{definition}

\vskip.125in 

\subsection{General measures} As we mention in the introduction (\ref{Multilinearfractalmama}), we may consider the fractal version of our multi-linear inequality. The following result was recently established by Shengze Duan (\cite{Duan2023}). 

\begin{theorem} (\cite{Duan2023}) Let 
$$(Af)(x):=\int_{\Phi(x,y)=0}f(y)\psi(x,y)d\sigma_x(y),$$ 
where $\sigma_x(y)$ is the surface measure on the submanifold defined by the zero set of $\Phi$ for every fixed $x$, and $\Phi(x,y)$ is assumed to be smooth on $(\R^d)^2$, and its gradients, $\nabla_x\Phi$ and $\nabla_y\Phi$, for both $R^d$-variables $x,y$, are assumed to be nonzero on the support of some fixed smooth compactly supported $\psi$. Suppose that $\mu (B(x,r))\leq C r^s$ and that both $A$ and its adjoint are bounded from $L^2(\R^d)$ to the $L^2$-based Sobolev space $L^2_\alpha(\R^d)$. Then, if $s>d-\alpha$, $A$ is bounded from $L^p(\mu)$ to $L^{p'}(\mu)$ for $\frac{1}{p}\in[\frac{1}{2},\frac{2s+2\alpha-2d+1}{2s+2\alpha-2d+2})$, and from $L^p(\mu)$ to $L^{2}(\mu)$ for $\frac{1}{p}\in[\frac{1}{2},\frac{3s+2\alpha-2d}{2s})$.

\vskip.125in 

Moreover, if $\mu$ further satisfies the lower bound $\mu (B(x,r))\geq c r^s$ for $x$ on its support, then $A$ is bounded from $L^p(\mu)$ to $L^p(\mu)$ for $\frac{1}{p}\in[\frac{1}{2},\frac{s+2\alpha-d}{2\alpha})$. In general, exact endpoint values aside, the threshold for $s$ is sharp, and the Riesz diagram is complete with the $L^p$-$L^p$ and the $L^p$-$L^{p'}$-endpoint.
\end{theorem}

Using this result as a springboard, it should be possible to obtain suitable variants of Theorem \ref{allgraphsgood}, Theorem \ref{thm:triangle}, and Theorem \ref{thm:2-chain}. Moreover, Theorem \ref{maintree}, Theorem \ref{treeremoval}, and Theorem \ref{joininggraphs} should hold in a general setting of measures. 




\end{document}